\documentclass[graybox]{svmult}

\usepackage{mathptmx}       
\usepackage{helvet}         
\usepackage{courier}        
\usepackage{type1cm}        
\usepackage{makeidx}         
\usepackage{graphicx}        
\usepackage{multicol}        
\usepackage[bottom]{footmisc}

\usepackage{lipsum}
\usepackage{amsfonts}
\usepackage{amsmath}
\usepackage{amssymb}
\usepackage{graphicx}
\usepackage{xcolor}
\usepackage{epstopdf}
\usepackage{algorithmic}
\usepackage{dsfont}
\usepackage{mathtools}
\usepackage{multirow}

\usepackage{pdflscape}
\usepackage{tikz}

\makeindex             

\begin{document}

\title*{Periodic homogenization of a pseudo-parabolic equation via a spatial-temporal decomposition}
\titlerunning{Upscaling of pseudo-parabolic equation via spatial-temporal decomposition}
\author{Arthur. J. Vromans and
        Fons van de Ven and
        Adrian Muntean}
\institute{
A.J. Vromans and A.A.F. van de Ven \at Centre for Analysis, Scientific computing and Applications (CASA), Eindhoven University of Technology,
              Den Dolech 2, 5612AZ, Eindhoven, The Netherlands, \email{a.j.vromans@tue.nl; a.a.f.v.d.ven@tue.nl}
\and A. Muntean \at Department of Mathematics and Computer Science, Karlstad University, Universitetsgatan 2, 651 88, Karlstad, Sweden, \email{adrian.muntean@kau.se}}
%
%
\maketitle

\abstract{Pseudo-parabolic equations have been used to model unsaturated fluid flow in porous media. In this paper it is shown how a pseudo-parabolic equation can be upscaled when using a spatio-temporal decomposition employed in the Peszy\'{n}ska-Showalter-Yi paper \cite{Pesz-Showalter-Yi2009}. The spatial-temporal decomposition transforms the pseudo-parabolic equation into a system containing an elliptic partial differential equation and a temporal ordinary differential equation. To strengthen our argument, the pseudo-parabolic equation has been given advection/convection/drift terms. The upscaling is done with the technique of periodic homogenization via two-scale convergence. The well-posedness of the extended pseudo-parabolic equation is shown as well. Moreover, we argue that under certain conditions, a non-local-in-time term arises from the elimination of an unknown.}
\begin{acknowledgement}
We acknowledge the Netherlands Organisation of Scientific Research (NWO) for the MPE grant 657.000.004 and we acknowledge the NWO Cluster Nonlinear Dynamics in Natural Systems (NDNS+) for funding a research stay of AJV at Karlstads Universitet.
\end{acknowledgement}
\section{Introduction}
\label{intro}
Groundwater recharge and pollution prediction for aquifers need models for describing unsaturated fluid flow in porous media. Pseudo-parabolic equations were found to be adequate models, see eqn. 25 in \cite{Hassanizadeh2002}. In \cite{Pesz-Showalter-Yi2009} a spatial-temporal decomposition of a pseudo-parabolic system was introduced. It was shown that this decomposition made upscaling of this system rather straightforward in several classical situations such as vanishing time-delay and double-porosity systems. In \cite{Pesz-Showalter-Yi2009} a toy pseudo-parabolic model was derived from a balance equation describing flow through a partially saturated porous medium. In our framework, a convective term that was dropped in \cite{Pesz-Showalter-Yi2009}, is retained in order to show that this term yields no additional problems for upscaling with the spatial-temporal decomposition. We want to convey the message that this decomposition can be applied not only to the physical system in \cite{Pesz-Showalter-Yi2009} but also to other physical systems with pseudo-parabolic equations, such as the concrete corrosion reaction model introduced in  \cite{VromansLIC}. Both these pseudo-parabolic systems are physical systems on a spatial micro scale with an intrinsic microscopic periodicity of size $\epsilon\ll1$. Similar intrinsic microscopic periodic behaviors are found in highly active research fields using composite structures or nano-structures.\\
$\;$\\
In this paper, we use this spatial-temporal decomposition to upscale our pseudo-parabolic equation by using the concept of periodic homogenization via two-scale convergence, which leads to a homogenized system that retains the spatial-temporal decomposition. We start in Section \ref{sec: 2} with formulating our pseudo-parabolic system \textbf{(Q$^\epsilon$)}, the decomposition system \textbf{(P$^\epsilon$)} and stating our assumptions. In Section \ref{H: sec: 4}, an existence and uniqueness result for weak solutions to our problem \textbf{(P$^\epsilon$)} is derived. In Section \ref{H: sec: 5}, we apply the idea of two-scale convergence to a weak version of problem \textbf{(P$^\epsilon$)}, denoted \textbf{(P$^\epsilon_w$)}, that contains the microscopic information at the $\epsilon$-level. Furthermore in this section, an upscaled system \textbf{(P$^0_w$)} of the weak system \textbf{(P$^\epsilon_w$)} is derived in the limit $\epsilon\downarrow0$, and, under certain conditions, an upscaled strong system \textbf{(P$^0_s$)} is obtained after eliminating several variables. This upscaled strong system contains a non-local-in-time term, but the system has lost the partial differential equation framework as a consequence. Contrary, the upscaled weak system \textbf{(P$^0_w$)} keeps the partial differential equation framework due to the spatial-temporal decomposition.
\section{Basic system and assumptions}
\label{sec: 2}
Our pseudo-parabolic system \textbf{(Q$^\epsilon$)} consists of a family of $N$ partial differential equations for the variable vector $\vec{U}^\epsilon(t,\vec{x},\vec{x}/\epsilon) = (U^\epsilon_1,\ldots,U^\epsilon_\alpha,\ldots,U^\epsilon_N)$ with $t>0$ and $\vec{x}=(x_1,\ldots,x_i,\ldots,x_d)\in\Omega\subset\mathbf{R}^d$.
For $\epsilon\in(0,\epsilon_0)$ with $\epsilon_0>0$, system \textbf{(Q$^\epsilon$)} is formulated as
\begin{equation*}
    \text{\textbf{(Q$^\epsilon$)}}\quad\begin{dcases}
    \tens{M}^\epsilon\tens{G}^{-1}\partial_t \vec{U}^\epsilon-\nabla\cdot\left((\tens{E}^\epsilon\cdot\nabla+\tens{D}^\epsilon)\tens{G}^{-1}(\partial_t\vec{U}^\epsilon+\tens{L}\vec{U}^\epsilon)\right)\cr
    \qquad= \vec{H}^\epsilon + (\tens{K}^\epsilon-\tens{M}^\epsilon\tens{G}^{-1}\tens{L})\vec{U}^\epsilon+\tens{J}^\epsilon\cdot\nabla\vec{U}^\epsilon&\text{on }\mathbf{R}_+\times\Omega,\cr
    \vec{U}^\epsilon= \vec{U}_*&\text{on }\{0\}\times\Omega,\cr
    \partial_t\vec{U}^\epsilon+\tens{L}\vec{U}^\epsilon = \vec{0}&\text{on }\mathbf{R}_+\times\partial\Omega.
    \end{dcases}\nonumber
\end{equation*}
The vectors $\vec{V}^\epsilon$ and $\vec{U}^\epsilon$ are both functions of the time coordinate $t$, the global or macro position coordinate $\vec{x}$, and also periodic functions of the micro (or nano) coordinate $\vec{y}\in Y$, where $\vec{y} = \vec{x}/\epsilon$, where the size of the micro domain $Y$ is $\mathcal{O}(\epsilon)$ of the size of the macro domain $\Omega$.\\
Our dimensionless decomposition system \textbf{(P$^\epsilon$)} consists of a family of $N$ partial differential equations (PDEs) and a family of $N$ ordinary differential equations (ODEs) for the two variable vectors  $\vec{V}^\epsilon(t,\vec{x},\vec{x}/\epsilon)=(V^\epsilon_1,\ldots,V^\epsilon_\alpha,\ldots,V^\epsilon_N)$ and $\vec{U}^\epsilon(t,\vec{x},\vec{x}/\epsilon)$. For $\epsilon\in(0,\epsilon_0)$ with $\epsilon_0>0$, it is formulated as
\begin{equation*}
    \text{\textbf{(P$^\epsilon$)}}\quad\begin{dcases}
    \tens{M}^\epsilon\vec{V}^\epsilon-\nabla\cdot\left(\tens{E}^\epsilon\cdot\nabla\vec{V}^\epsilon+\tens{D}^\epsilon\vec{V}^\epsilon\right)= \vec{H}^\epsilon + \tens{K}^\epsilon\vec{U}^\epsilon+\tens{J}^\epsilon\cdot\nabla\vec{U}^\epsilon&\text{on }\mathbf{R}_+\times\Omega,\cr
    \partial_t\vec{U}^\epsilon+\tens{L}\vec{U}^\epsilon= \tens{G}\vec{V}^\epsilon&\text{on }\mathbf{R}_+\times\Omega,\cr
    \vec{U}^\epsilon= \vec{U}_*&\text{on }\{0\}\times\Omega,\cr
    \vec{V}^\epsilon = \vec{0}&\text{on }\mathbf{R}_+\times\partial\Omega.
    \end{dcases}\nonumber
\end{equation*}
Above, the $\epsilon$-dependent notation $c^\epsilon(t,\vec{x}) = c(t,\vec{x},\vec{x}/\epsilon)$ is used for the $\epsilon$-independent 1-,2- and 3-tensors of assumption (A1).
\begin{itemize}
    \item[(A1)] \qquad For all $\alpha,\beta\in\{1,\ldots,N\}$ and for all $i,j\in\{1,\ldots,d\}$, we have
    \begin{eqnarray*}
    \tens{M}_{\alpha\beta},\,\tens{E}_{ij},\,\tens{D}_{i\alpha\beta},\,\vec{H}_\alpha,\,\tens{K}_{\alpha\beta},\,\tens{J}_{i\alpha\beta}&\in&L^\infty(\mathbf{R}_+\times\Omega;C_\#(Y)),\cr
    \tens{L}_{\alpha\beta},\,\tens{G}_{\alpha\beta}&\in&L^\infty(\mathbf{R}_+;W^{1,\infty}(\Omega)),\cr
    \vec{U}_*&\in& C^1(\Omega)^N,
    \end{eqnarray*}
    with $\tens{G}$ invertible.
    \item[(A2)] \qquad Let the tensors $\tens{M}^\epsilon$ and $\tens{E}^\epsilon$ be in diagonal form\footnote{Due to the Theorem of Jacobi about quadratic forms (cf. \cite{Lam1999}) in combination with the coercivity of both $\tens{M}^\epsilon$ and $\tens{E}^\epsilon$, we are allowed to assume diagonal forms of $\tens{M}^\epsilon$ and $\tens{E}^\epsilon$ as the orthogonal transformations, necessary to put their quadratic forms in diagonal form, modify the domain $\Omega^\epsilon$ and the coefficients of $\tens{D}^\epsilon$, $\vec{H}^\epsilon$, $\tens{K}^\epsilon$ and $\tens{J}^\epsilon$ without changing their regularity.} with elements $m^\epsilon_\alpha>0$ and $e^\epsilon_i>0$, respectively, satisfying $1/m_\alpha^\epsilon,1/e_i^\epsilon\in L^\infty(\mathbf{R}_+\times\Omega;C_\#(Y))$.
    \item[(A3)] \qquad The inequality
    \begin{equation*}
        \|\tens{D}_{i\beta\alpha}^\epsilon\|_{L^\infty(\mathbf{R}_+\!\times\Omega^\epsilon;C_\#(Y))}^2\!<\!\frac{4}{dN^2\!\left\|\!1/m_\alpha^\epsilon\!\right\|_{L^\infty(\mathbf{R}_+\!\times\Omega^\epsilon;C_\#(Y))}\!\left\|\!1/e_i^\epsilon\!\right\|_{L^\infty(\mathbf{R}_+\!\times\Omega^\epsilon;C_\#(Y))}}\label{H: eq:  ineq}
    \end{equation*}
    holds for all $\alpha,\beta\in\{1,\ldots,N\}$, for all $i\in\{1,\ldots,n\}$, and for all $\epsilon\in(0,\epsilon_0)$.
\end{itemize}
    Remark, inequality (\ref{H: eq:  ineq}) implies that automatically (\ref{H: eq:  ineq}) holds for the $Y$-averaged functions $\overline{\tens{D}_{i\beta\alpha}^\epsilon}$, $\overline{\tens{M}_{\beta\alpha}^\epsilon}$, and $\overline{\tens{E}^\epsilon_{ij}}$ in $L^\infty(\mathbf{R}_+\times\Omega)$, using $|Y|\overline{f}(t,\vec{x}) = \int_{Y}f(t,\vec{x},\vec{y})\mathrm{d}\vec{y}$.

\section{Existence and uniqueness of weak solutions to \textbf{(P$^\epsilon_w$)}}\label{H: sec: 4}
In this section, we show the existence and uniqueness of a weak solution $(\vec{U},\vec{V})$ to \textbf{(P$^\epsilon$)}. We define a weak solution to \textbf{(P$^\epsilon$)} for $\epsilon\in(0,\epsilon_0)$ and $T\in\mathbf{R}_+$ as a pair of sequences $(\vec{U}^\epsilon,\vec{V}^\epsilon)\in H^1((0,T)\times\Omega)^N\times L^\infty((0,T),H_0^1(\Omega))^N$ satisfying
\begin{equation*}
    \text{\textbf{(P$^\epsilon_w$)}}
    \begin{dcases}
    \int_\Omega\vec{\phi}^\top\left[\tens{M}^\epsilon\vec{V}^\epsilon\!\!-\!\vec{H}^\epsilon\!\!-\!\tens{K}^\epsilon\vec{U}^\epsilon\!\!-\!\tens{J}^\epsilon\!\cdot\!\nabla\vec{U}^\epsilon\right]
    \!+\!(\nabla\phi)^\top\!\!\cdot\!\left(\tens{E}^\epsilon\!\!\cdot\!\nabla\vec{V}^\epsilon\!\!+\!\tens{D}^\epsilon\vec{V}^\epsilon\right)\mathrm{d}\vec{x}\!=\!0,\cr
    \int_\Omega\vec{\psi}^\top\left[\partial_t\vec{U}^\epsilon+\tens{L}^\epsilon\vec{U}^\epsilon- \tens{G}^\epsilon\vec{V}^\epsilon\right]\mathrm{d}\vec{x} = 0,\cr
    \vec{U}^\epsilon(0,\vec{x})=\vec{U}_*(\vec{x})\text{ for all }\vec{x}\in\Omega,
    \end{dcases}
\end{equation*}
for a.e. $t\in(0,T)$, for all test-functions $\vec{\phi}\in H^1_0(\Omega)^N$ and $\vec{\psi}\in L^2(\Omega)^N$.\\
The existence and uniqueness can only hold when the first equation of \text{\textbf{(P$^\epsilon_w$)}}  satisfies all the conditions of Lax-Milgram. The next lemma provides the coercivity condition, while the continuity condition is trivially satisfied.
\begin{lemma}\label{H: l: : apriori} Assume assumptions (A1) - (A3) hold, then there exist positive constants $\tilde{m}_\alpha$, $\tilde{e}_i$, $\tilde{H}$, $\tilde{K}_\alpha$, $\tilde{J}_{i\alpha}$ for $\alpha\in\{1,\ldots,N\}$ and $i\in\{1,\ldots,d\}$ such that the following a-priori estimate holds for a.e. $t\in(0,T)$.
    \begin{eqnarray}
        \sum_{\alpha=1}^N\tilde{m}_\alpha\|\vec{V}_\alpha^\epsilon\|^2_{L^2(\Omega)}&+&\sum_{i=1}^d\sum_{\alpha=1}^N\tilde{e}_i\|\partial_{x_i}\vec{V}_\alpha^\epsilon\|^2_{L^2(\Omega)}\cr
        &\leq& \tilde{H}+\sum_{\alpha=1}^N\tilde{K}_\alpha\|\vec{U}_\alpha^\epsilon\|^2_{L^2(\Omega)}+\sum_{i=1}^d\sum_{\alpha=1}^N\tilde{J}_{i\alpha}\|\partial_{x_i}\vec{U}_\alpha^\epsilon\|^2_{L^2(\Omega)}\label{H: eq:  aprioriV}
    \end{eqnarray}
\end{lemma}
\begin{proof}
See pages 92, 93 in \cite{VromansLIC} for proof and relation with parameters of \text{\textbf{(P$^\epsilon_w$)}}.
\qed\end{proof}
\begin{theorem}\label{H: t: existuniq} Assume assumptions  (A1) - (A3) hold, then there exists a unique pair $(\vec{U}^\epsilon,\vec{V}^\epsilon)\in H^1((0,T)\times\Omega)^N\times L^\infty((0,T),H_0^1(\Omega))^N$ such that $(\vec{U}^\epsilon,\vec{V}^\epsilon)$ is a weak solution to \textbf{(P$_w^\epsilon$)}.
\end{theorem}
\begin{proof}
Use $\vec{\phi} = \vec{V}^\epsilon$ and apply Lemma \ref{H: l: : apriori}. Then use $\vec{\psi} \in \{\vec{U}^\epsilon,\partial_t\vec{U}^\epsilon\}$. Moreover, apply a gradient to the second equation of \textbf{(P$^\epsilon$)} and test that equation with $\nabla\vec{U}^\epsilon$ and $\partial_t\nabla\vec{U}^\epsilon$. Application of Young's inequality, use of (\ref{H: eq:  aprioriV}) and application of Gronwall's inequality, see \cite[Thm. 1]{Dragomir2003}, yields the existence for $\vec{U}^\epsilon$. Then Lax-Milgram yields the existence for $\vec{V}^\epsilon$. Uniqueness follows from the bilinearity of \textbf{(P$_w^\epsilon$)}. For more details, see pages 93 and 94 in \cite{VromansLIC}.
 \qed\end{proof}
\section{Upscaling the system (P$^\epsilon_w$) via two-scale convergence}
\label{H: sec: 5}
Based on two-scale convergence, see \cite{Allaire1992}, \cite{LukkassenNguetsengWall2002}, \cite{Nguetseng1989} for details, we obtain the following Lemma ensuring that the weak solution to problem \textbf{(P$^\epsilon_w$)} has two-scale limits in the limit $\epsilon\downarrow0$.
\begin{lemma}
    \label{H: l: : two-scale-conv} Assume assumptions (A0), (A1), (A2) to hold. For each $\epsilon\in(0,\epsilon_0)$, let the pair of sequences $(\vec{U}^\epsilon,\vec{V}^\epsilon)\in H^1((0,T)\times\Omega)\times L^\infty((0,T);H^1_0(\Omega))$ be the unique weak solution to \textbf{(P$^{\,\epsilon}_w$)}. Then this sequence of weak solutions satisfies the estimate $
    \|\vec{U}^\epsilon\|_{H^1((0,T)\times\Omega)^N}+\|\vec{V}^\epsilon\|_{L^\infty((0,T),H^1_0(\Omega))^N}\leq C,$
for all $\epsilon\in(0,\epsilon_0)$ and there exist vector functions
\begin{eqnarray}
    \vec{u}\text{ in }H^1((0,T)\times\Omega)^N,\quad\qquad\mathcal{U}&\text{ in }&H^1((0,T);L^2(\Omega;H^1_\#(Y)/\mathbf{R}))^N,\cr
    \vec{v}\text{ in }L^\infty((0,T);H^1_0(\Omega))^N,\qquad\mathcal{V}&\text{ in }&L^\infty((0,T)\times\Omega;H^1_\#(Y)/\mathbf{R})^N,\nonumber
\end{eqnarray}
and a subsequence $\epsilon'\subset\epsilon$, for which the following two-scale convergences
\begin{eqnarray}
    \vec{U}^{\epsilon'}\overset{2}{\longrightarrow}\!\quad\vec{u}(t,\vec{x}),
    &&\!\quad\nabla\vec{U}^{\epsilon'}\overset{2}{\longrightarrow}\!\!\quad\nabla\vec{u}(t,\vec{x})+\nabla_{\vec{y}}\mathcal{U}(t,\vec{x},\vec{y}),\cr
    \partial_t\vec{U}^{\epsilon'}\overset{2}{\longrightarrow}\partial_t\vec{u}(t,\vec{x}),
    &&\partial_t\nabla\vec{U}^{\epsilon'}\overset{2}{\longrightarrow}\partial_t\nabla\vec{u}(t,\vec{x})+\partial_t\nabla_{\vec{y}}\mathcal{U}(t,\vec{x},\vec{y}),\cr
    \vec{V}^{\epsilon'}\overset{2}{\longrightarrow}\!\quad\vec{v}(t,\vec{x}),
    &&\!\quad\nabla\vec{V}^{\epsilon'}\overset{2}{\longrightarrow}\!\!\quad\nabla\vec{v}(t,\vec{x})+\nabla_{\vec{y}}\mathcal{V}(t,\vec{x},\vec{y})\nonumber
\end{eqnarray}
 hold for a.e. $t\in(0,T)$.
\end{lemma}
\begin{proof}
See pages 95 and 96 of \cite{VromansLIC}.
\qed\end{proof}
Using Lemma \ref{H: l: : two-scale-conv}, we upscale \textbf{(P$^\epsilon_w$)} to \textbf{(P$^0_w$)} via two-scale convergence.
\begin{theorem}\label{H: t: upscale} Assume the conditions of Lemma \ref{H: l: : two-scale-conv} are met. Then the two-scale limits $\vec{u}\in H^1((0,T)\times\Omega)^N$, $\mathcal{U}\in H^1((0,T);L^2(\Omega;H^1_\#(Y)/\mathbf{R}))^N$ and $\vec{v}\in L^\infty((0,T);H^1_0(\Omega))^N$ introduced in Lemma \ref{H: l: : two-scale-conv} form the weak solution triple to \begin{equation*}
    \text{\textbf{(P$^0_w$)}}\qquad\quad\begin{dcases}
    \int_\Omega\vec{\phi}^\top\left[\overline{\tens{M}}\vec{v}-\overline{\vec{H}} - \overline{\tens{K}}\vec{u}-\overline{\tens{J}}\cdot\nabla\vec{u}-\frac{1}{|Y|}\int_Y\tens{J}\cdot\nabla_{\vec{y}}\mathcal{U}\mathrm{d}\vec{y}\right]&\cr
    \qquad\qquad+(\nabla\phi)^\top\cdot\left(\tens{E}^*\cdot\nabla\vec{v}+\tens{D}^*\vec{v}\right)\mathrm{d}\vec{x} = 0,\cr
    \int_\Omega\vec{\psi}^\top\left[\partial_t\vec{u}+\tens{L}\vec{u}- \tens{G}\vec{v}\right]\mathrm{d}\vec{x} = 0,\cr
    \int_Y\vec{\xi}^\top\cdot\nabla_{\vec{y}}\left[\partial_t\mathcal{U}+\tens{L}\mathcal{U}- \tilde{\delta}\vec{v}-\tilde{\omega}\cdot\nabla\vec{v}\right]\mathrm{d}\vec{y} = 0,\cr
    \vec{u}(0,\vec{x})=\vec{U}_*(\vec{x})\qquad\text{on }\Omega,\cr
    \nabla_{\vec{y}}\mathcal{U}(0,\vec{x},\vec{y}) = \vec{0}\qquad\text{on }\Omega\times Y,
    \end{dcases}
\end{equation*}
for a.e. $t\in(0,T)$, for all test-functions $\vec{\phi}\in H^1_0(\Omega)^N$, $\vec{\psi}\in L^2(\Omega)^N$, and $\vec{\xi}\in H^1_\#(Y)^{d\times N}$, where the \emph{effective} coefficients $\tens{E}^*$ and $\tens{D}^*$ are given by
\begin{eqnarray*}
\tens{E}^*=\frac{1}{|Y|}\int_Y\tens{E}\cdot(\tens{1}+\nabla_{\vec{y}}\vec{W})\mathrm{d}\vec{y},&&\qquad\tens{D}^*=\frac{1}{|Y|}\int_Y\tens{D}+\tens{E}\cdot\nabla_{\vec{y}}\tens{\delta}\mathrm{d}\vec{y},\cr
\tilde{\delta}=\nabla_{\vec{y}}(\tens{G}\delta),&&\qquad\tilde{\omega}=\nabla_{\vec{y}}\vec{W}\otimes\tens{G},
\end{eqnarray*}
and the tensor $\tens{\delta}_{\alpha\beta}\in L^\infty((0,T)\times\Omega;H^1_\#(Y)/\mathbf{R}))$ and vector $\vec{W}_i\in L^\infty((0,T)\times\Omega;H^1_\#(Y)/\mathbf{R}))$ satisfy the cell problems
\begin{subequations}
\begin{align}
0&=\int_Y\Phi^\top\cdot(\nabla_{\vec{y}}\cdot\left[\tens{E}\cdot(\tens{1}+\nabla_{\vec{y}}\vec{W})\right])\mathrm{d}\vec{y},&
0&=\int_Y\Psi^\top(\nabla_{\vec{y}}\cdot\left[\tens{D}+\tens{E}\cdot\nabla_{\vec{y}}\tens{\delta}\right])\mathrm{d}\vec{y}\nonumber
\end{align}\nonumber
\end{subequations}
for all $\Phi\in C_\#(Y)^d$, $\Psi\in C_\#(Y)^{N\times N}$.
\end{theorem}
\begin{proof} In \textbf{(P$^\epsilon_w$)}, we choose $\phi=\phi^\epsilon = \Phi\left(t,\vec{x},\frac{\vec{x}}{\epsilon}\right)$ and $\psi=\psi^\epsilon = \Psi(t,\vec{x})+\epsilon\vec{\varphi}\left(t,\vec{x},\frac{\vec{x}}{\epsilon}\right)$ for the test-functions $\Phi\in L^2((0,T);\mathcal{D}(\Omega;C^\infty_\#(Y)))^N$,
 $\Psi\in L^2((0,T);C^\infty_0(\Omega))^N$ and for $\vec{\varphi}\in L^2((0,T);\mathcal{D}(\Omega;C^\infty_\#(Y)))^N$. Two-scale convergence limits, see \cite{Allaire1992}, \cite{LukkassenNguetsengWall2002}, \cite{Nguetseng1989}, and cell-function arguments, see  \cite{MunteanChalupecky2011}, give \textbf{(P$^0_w$)}. Details in pages 97,98 of \cite{VromansLIC}.
 \qed\end{proof}
We have shown that upscaling system \textbf{(P$^\epsilon_w$)} yields system \textbf{(P$^0_w$)}. This system contains only PDEs with respect to $(t,\vec{x})$. However, an extra variable $\nabla_{\vec{y}}\mathcal{U}$ was needed. Removing $\nabla_{\vec{y}}\mathcal{U}$ needs the use of continuous semi-group theory, see papers 10 and 14 of \cite{Yosida1965}, for solving the third equation of system \textbf{(P$^0_w$)}. This leads to a non-local-in-time term as a consequence of removing $\nabla_{\vec{y}}\mathcal{U}$.
\section{Conclusion}
Our main goal of this paper is to show that the spatial-temporal decomposition, as employed in \cite{Pesz-Showalter-Yi2009}, allows for the straighforward upscaling of pseudo-parabolic equations, in specific for system \textbf{(Q$^\epsilon$)}. The upscaling procedure is here performed using the concept of two-scale convergence as reported in Section \ref{H: sec: 5}. Moreover, the decomposition is retained in the upscaled limit. A non-local-in-time term arose when an extra variable was eliminated. The spatial-temporal decoupling showed why this non-local term is non-local in time.\\
In future research we intend to investigate the applicability of the spatial-temporal decomposition of our pseudo-parabolic system to perforated periodic domains, corrector estimates (convergence speed estimate) and  high-contrast situations.
%
%

\end{document}